\documentclass{article}

\usepackage{amsmath}
\usepackage{amsthm}
\usepackage{amsfonts}
\usepackage{amssymb}
\usepackage{graphicx}

\newtheorem{thm}{Theorem}
\newtheorem{cor}{Corollary}
\newtheorem{lem}{Lemma}
\newtheorem{prop}{Proposition}
\theoremstyle{definition}

\theoremstyle{remark}

\newtheorem{ob}{Observation}

\newcommand{\R}{{\mathbb{R}}}

\begin{document}

\title{Knotting and linking in the Petersen family}         
\author{Danielle O'Donnol}        
\date{}          
\maketitle

\begin{abstract}  This paper focuses on the graphs in the Petersen family, the set of minor minimal intrinsically linked graphs.  We prove there is a relationship between algebraic linking of an embedding and knotting in an embedding.  We also present a more explicit relationship for the graph $K_{3,3,1}$ between knotting and linking, which relates the sum of the squares of linking numbers of links in the embedding and the second coefficient of the Conway polynomial of certain cycles in the embedding.  
\end{abstract}

\section{Introduction}
Throughout this paper we will work with finite simple graphs, in the piecewise linear category.  A \emph{spatial graph} is an embedding of a graphs $G$ into $\R^3$ or $S^3$, denoted $f(G)$ or simply $f$.  This paper focuses on 
the interaction between knotting and linking in spatial graphs.  A link is said to be in a spatial graph if the link appears as a set of the embedded cycles.  An embedding $f$ of a graph $G$ is  \emph{linked} 
if there is a nontrivial link in $f(G).$  An 
embedding $f$ of a graph $G$ is \emph{algebraically linked} if there is a  link with nonzero linking number 
in $f(G).$ We will say an embedding $f$ of a graph $G$ is \emph{complexly algebraically linked} (\emph{CA linked}) if there is a link $L$ in $f(G)$ 
where $|lk(L)|\geq 2$ or if there is more than one link in $f(G)$ with nonzero linking number.  
An embedding $f$ of a graph $G$ is  \emph{knotted} if there 
is a nontrivial knot in $f(G).$  An embedding that is not knotted is called \emph{knotless}.  

\begin{figure}[h]
\begin{center}
\includegraphics[scale=0.5]{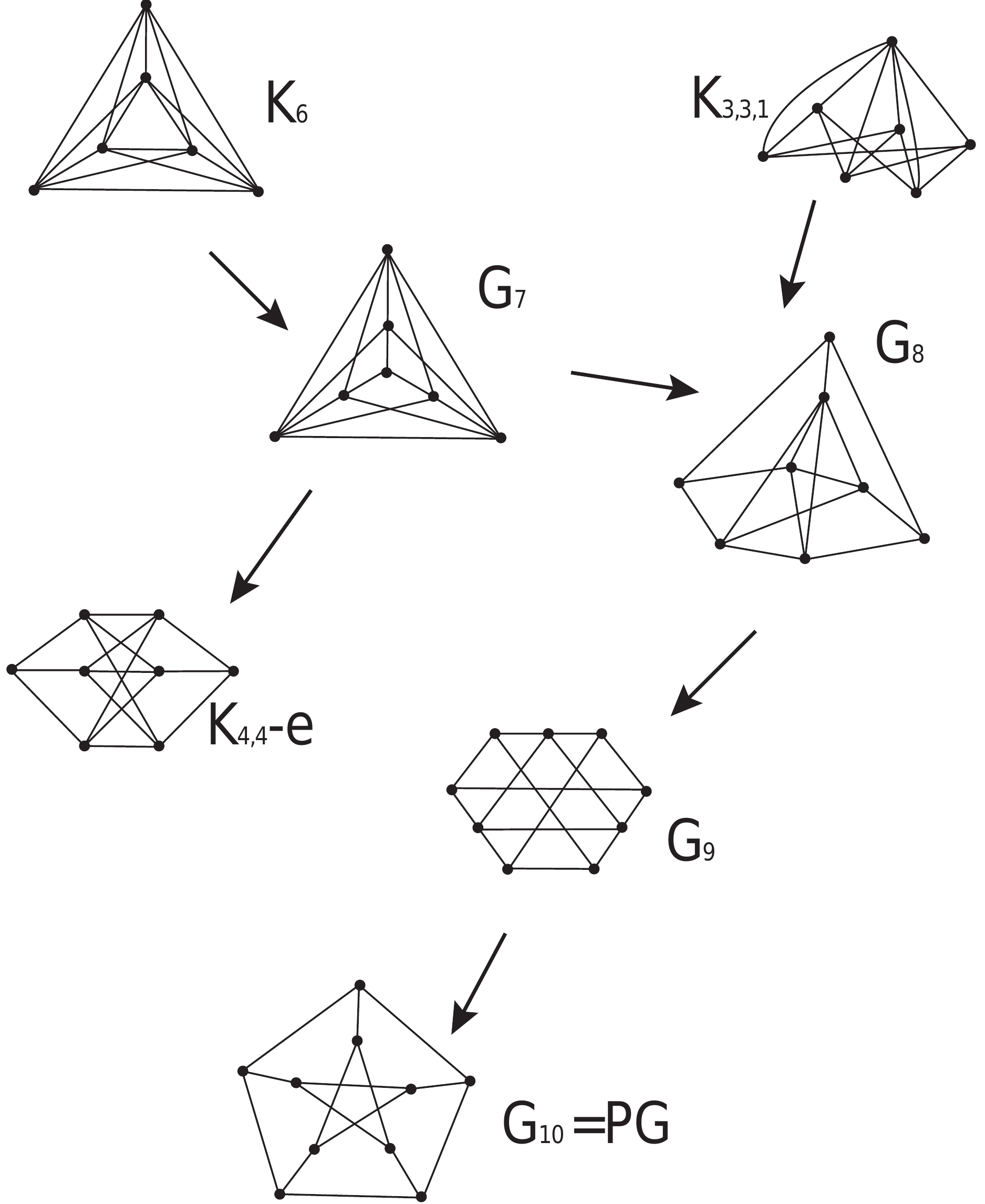}
\caption{The graphs of the Petersen family.  The arrows indicate a $\nabla$Y-move.  }\label{PF}
\end{center}
\end{figure}

A graph, $G$, is \emph{intrinsically knotted} if every embedding of
$G$ into $\R^3$ or $S^3$ contains a nontrivial knot.  A graph, $G$, is \emph{intrinsically linked} if every embedding of
$G$ into $\R^3$ or $S^3$ contains a non-split link.  The combined work of Conway and Gordon
\cite{CG}, Sachs \cite{Sa}, and  Robertson, Seymour, and
Thomas \cite{RST} fully characterize intrinsically linked graphs.
They showed that the Petersen family is the complete set of 
minor minimal intrinsically linked graphs, thus any intrinsically linked graph
contains a graph in the Petersen family as a minor. The Petersen family is a set of seven graphs shown in Figure \ref{PF}.  They are related by $\nabla$Y-moves (Figure \ref{TY}) as indicated by the arrows in Figure \ref{PF}.  We will denote this set of graphs by $\mathcal{PF}$.  
The set of intrinsically knotted graphs has not been fully characterized, however it is known that every intrinsically knotted graph is intrinsically linked.  This is a consequence of the work on characterizing intrinsically linked graphs \cite{RST}.  The converse does not hold, there are many graphs that are intrinsically linked graphs that have knotless embeddings.  In particular, none of the graphs of $\mathcal{PF}$ are intrinsically knotted.

In this paper we will examine the relationship between knotting and linking in the Petersen family.  One might expect that a knotted embedding would be an embedding with more complex linking.  However there are knotted embeddings of $K_6$ that contain only a single Hopf link, see Figure \ref{K6}.  The question of when complexity in linking of any embedding can guaranty that the embedding is knotted, is much more fruitful.  
We prove:
\begin{thm}\label{main} If $f$ is a CA linked embedding of $G\in\mathcal{PF}$, then $f(G)$ is knotted.  
\end{thm}
\noindent This result gives an algebraic linking condition on the embedding that will result in a knotted embedding.  Another natural question is whether the presences of additional links with linking number 0, or more complex links with linking number 1 would guarantee in a knotted embedding.  In Section \ref{eg}, we give examples of embeddings of $K_6$ suggesting that such geometric linking will not guarantee a knotted embedding.

Theorem \ref{main} rests on understanding the interactions between linking and knotting in $\mathcal{PF}$.  In keeping with the notation of Nikkuni \cite{Ni}, let $\Gamma(G)$ denote 
the set of all cycles (or simple closed curves) in $G$, let $\Gamma_H(G)$ be the set of all Hamiltonian cycles 
in $G$, let $\Gamma_m(G)$ be the set of all $m$-cycles in $G$, 
let $\Gamma_{s,t}(G)$ be the set of all pairs of disjoint $s$-cycles and $t$-cycles, and let $\Lambda(G)$ be the set of all 
pairs of disjoint cycles.  Recently, Nikkuni proved the following theorem relating the linking and knotting in an embedding of $K_6\in\mathcal{PF}$:

\begin{thm} \label{nikkuni}\textup{\cite{Ni}} For any embedding $f$ of $K_6$ into $\R^3$ or $S^3$ the following holds:
$$\sum_{\lambda\in\Lambda(K_6)} lk(f(\lambda))^2= 2\Big(\sum_{\gamma\in\Gamma_H}a_2(f(\gamma))-\sum_{\gamma\in\Gamma_5}a_2(f(\gamma))\Big)+1.$$
\end{thm}

\noindent Following similar methods, in Section \ref{wu}, Theorem \ref{k331} we obtain a similar result for the graph $K_{3,3,1}$.  We show for every embedding $f$ of $K_{3,3,1}$ that $$\sum_{\lambda\in\Gamma_{3,4} (K_{3,3,1})} lk(f(\lambda))^2=2\Big(\sum_{\gamma\in\Gamma_H}a_2(f(\gamma))-2\sum_{\substack{\gamma\in\Gamma_6\\A\notin\gamma}}a_2(f(\gamma))-\sum_{\substack{\gamma\in\Gamma_5\\A\in\gamma}}a_2(f(\gamma))\Big)+1,$$  
where $A$ is the single vertex of valance 9 in $K_{3,3,1}$.  This gives an explicit connection between linking and knotting in embeddings of $K_{3,3,1}\in\mathcal{PF}.$

\smallskip
\noindent {\bf Acknowledgements:} The author would like to thank Ryo Nikkuni, Kouki Taniyama, and Tim Cochran for many useful conversations.  


\section{Graph homologous embeddings and the Wu Invariant}\label{wu}
This sections contains a brief introduction to the Wu invariant, and graph-homologous embeddings.  Then these tools, along with useful relationships between the Wu invariant, the $\alpha-invariant$, and the second coefficient of the Conway polynomial, are used to obtain Theorem \ref{k331}, relating the linking and knotting in embeddings of $K_{3,3,1}$.  

Let $G$ be a graph with $E(G)=\{e_1,\dots, e_n\}$ and $V(G)=\{v_1,\dots, v_m\}$ (fixed ordering), and a fixed orientation on each of the edges.  Note, $G$ is a finite one-dimensional simplicial complex.  For a simplicial complex $X,$ let $P_2(X)=\{s_1\times s_2|s_1,s_2\in X, s_1\cap s_2=\emptyset\}$ be the \emph{polyhedral residual space} of $X$.    Let $\sigma$ be the involution on $P_2(X)$, i.e. $\sigma(s_1\times s_2)=s_2\times s_1$.  Let $f$ be an embedding of $G$ into $\R^3$.  The Wu invariant of $f$, denoted $\mathcal{L}(f)$ is in the second skew-symmetric cohomology group $H^2(P_2(G),\sigma),$ which we will denote $L(G).$  For more background on the Wu invariant and a more general approach see \cite{RN, Ta, T, Wu}

Following \cite{T}, Section 2, there is explicit presentation of $L(G).$  An orientation of a 2-cell $e_i\times e_j\in P_2(G)$ is given by the ordered pair of orientations of $e_i$ and $e_j$.  Let $E_{e_ie_j}=e_i\times e_j+e_j\times e_i\in C_2(P_2(G))$ for $e_i\cap e_j=\emptyset$ $(1\leq i<j \leq n).$  The set $\{E_{e_ie_j}|1\leq i<j \leq n, e_i\cap e_j=\emptyset\}$ is a free basis for $C_2(P_2(G),\sigma).$ Now the set of dual elements $\{E^{e_ie_j}|1\leq i<j \leq n, e_i\cap e_j=\emptyset\}$ generate $L(G)$.   The relations are given by the coboundary applied to the set $\{V^{e_iv_s}|1\leq i \leq n, 1\leq s \leq m,  v_s\not\in e_i\}.$  The coboundary is defined by: 
$$\delta^1(V^{e_iv_s})=\sum_{I(e_j)=v_s}E^{\rho(e_ie_j)}-\sum_{T(e_j)=v_s}E^{\rho(e_ie_j)},$$ 
where $I(e_i)$ is the initial vertex of $e_j$, $T(e_j)$ is the terminal vertex of $e_j$ and $\rho(e_ie_j)$ is the standard ordering $e_ie_j$ if $i<j$ and $e_je_i$ if $j<i$.  The Wu invariant $\mathcal{L}(f)$ can be calculated from a projection $\lambda :\R^3\to\R^2$ where $\lambda\circ f$ is a regular projection with finitely many multiple points all of which are transverse double points that occur away from vertices.  Let $a_{ij}(f)$ be the sum of the signs of the crossings that occur between $\lambda\circ f(e_i)$ and $\lambda\circ f(e_j)$, the Wu invariant $\mathcal{L}(f)$ is the coset of $\sum a_{ij}(f)E^{e_ie_j}$ in $L(G),$ which is summed over all pairs of disjoint edges of $G$.

\begin{figure}[h]
\begin{center}
\includegraphics[scale=0.7]{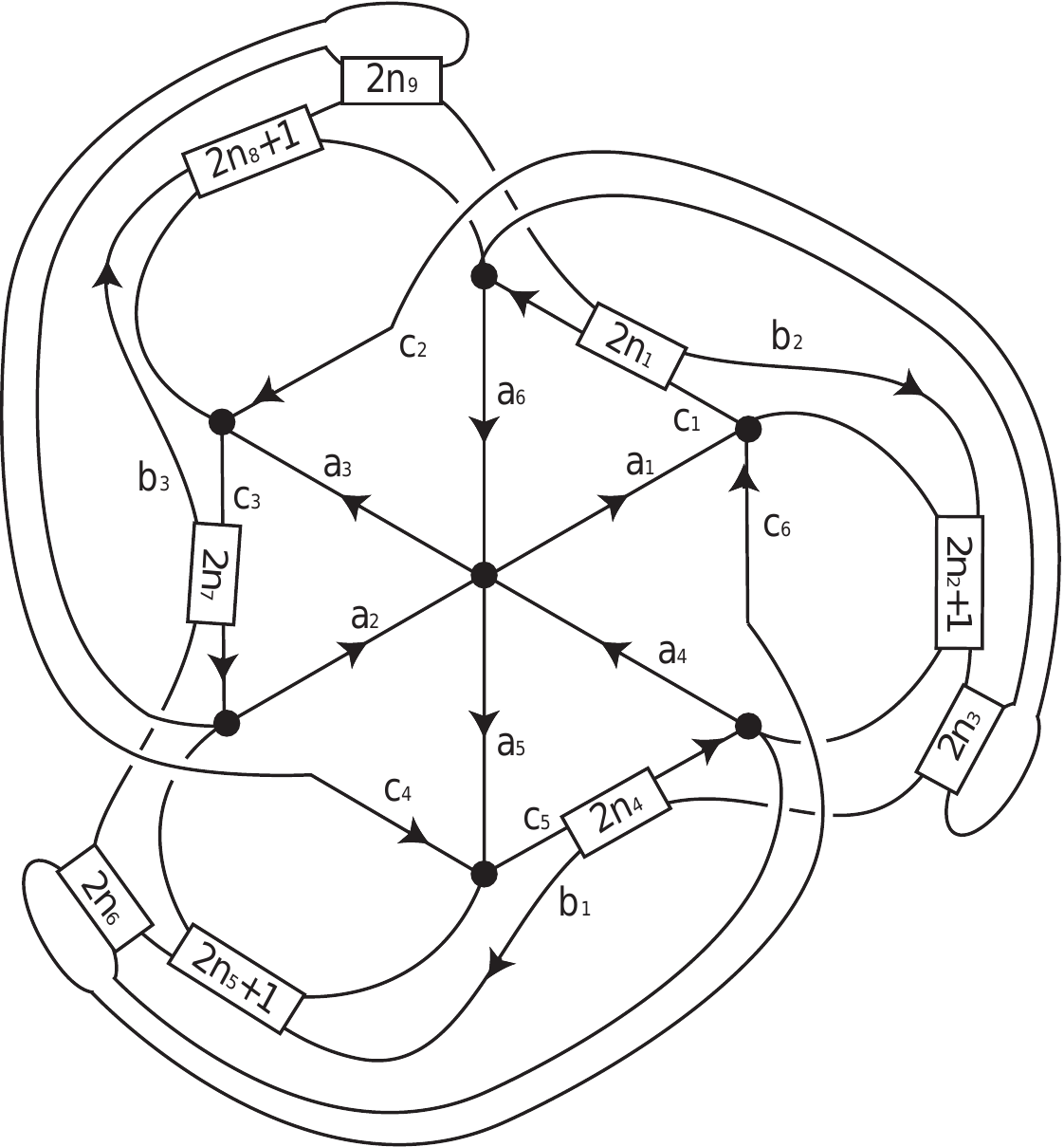}
\caption{An embedding $h$ of $K_{3,3,1}$ where the integers in the boxes indicate the number of half twists between the two edges.  }\label{K331h}
\end{center}
\end{figure}
\begin{figure}[h]
\begin{center}
\includegraphics[scale=0.6]{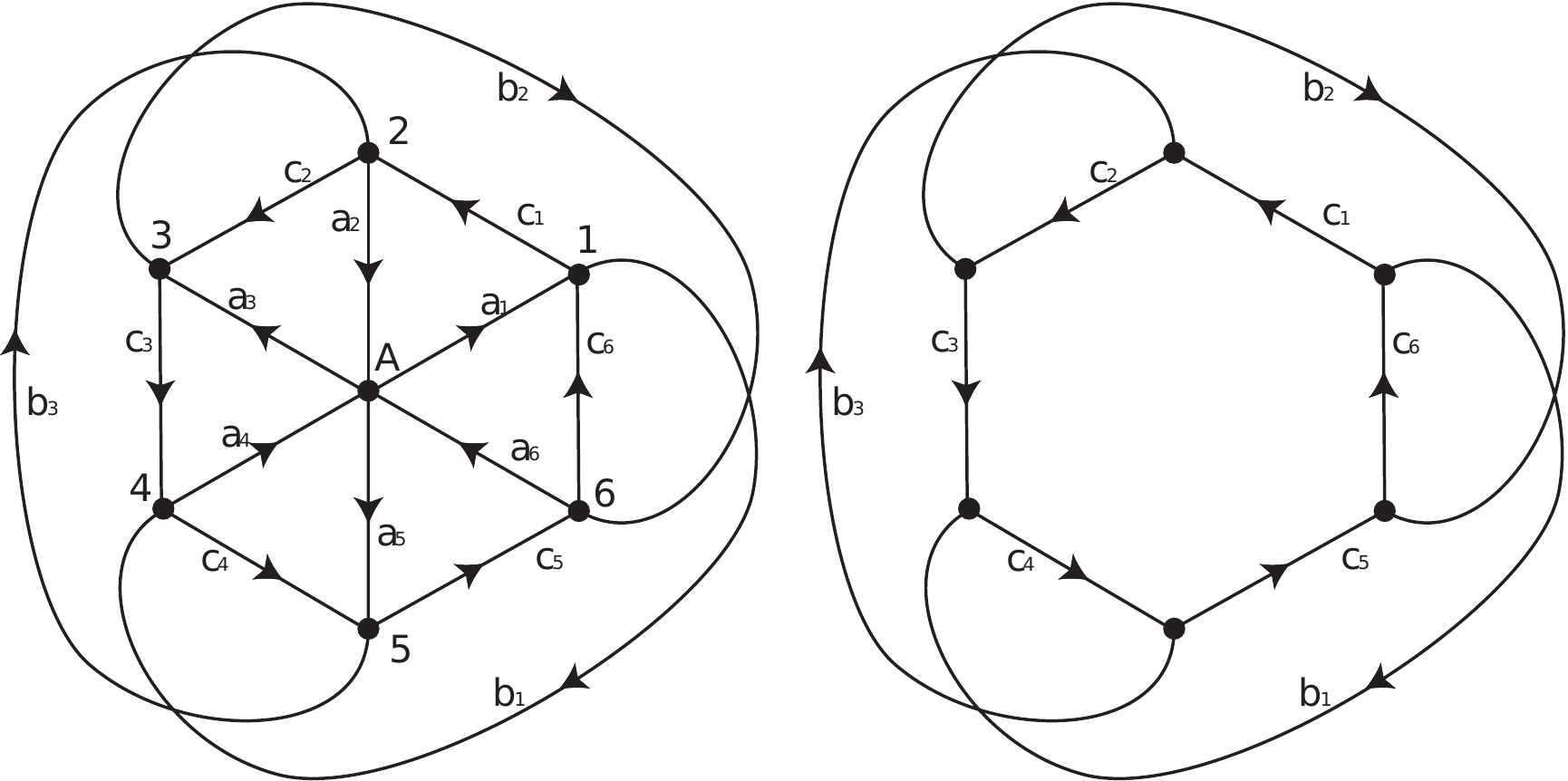}
\caption{On the left:  The graph $K_{3,3,1}$ with edges oriented and the edges and vertices labeled.  On the right:  The graph $K_{3,3}$ with edges oriented and labeled in the standard convention for the Wu invariant.  }\label{K33}
\end{center}
\end{figure}

Two embeddings $f,g$ of a graph $G$ are \emph{spatial graph-homologous} (or just \emph{homologous}) if there is a locally flat embedding $\Phi : (G\times I)\sharp \bigcup S_i\rightarrow\R^3\times I$ with $\Phi(G\times \{0\})\in\R^3\times \{0\}$ and  $\Phi(G\times \{1\})\in\R^3\times \{1\}$ where $S_i$ is a closed orientable surface and $S_i$ is attached on $Int(e\times I)$ for an edge $e\in E(G)$ by connected sum.   In \cite{T}, Taniyama showed the following:
\begin{thm}\label{emb} Two embeddings $f$ and $g$ of a simple graph $G$ into $\R^3$ are homologous if and only if $\mathcal{L}(f)=\mathcal{L}(g)$.  
\end{thm}

\begin{prop}\label{wu331}
For every embedding of $K_{3,3,1}$ there are nine integers $n_1,\dots, n_9$ such that $f$ is spatial graph-homologous to the embedding $h$ of  $K_{3,3,1}$ shown in Figure \ref{K331h}.
\end{prop}

\begin{proof} We will use the edge and vertex labeling, as well as edge orientation indicated in Figure \ref{K33}.  The order on the sets is as given $E(K_{3,3,1})=\{a_1, \dots, a_6, b_1, b_2, \\b_3, c_1, \dots , c_6\}$ and $V(K_{3,3,1})=\{1,2,3,4,5,6,A\}.$  By Theorem \ref{emb} we need only show that $\mathrm{L(G)}$ can be generated by the set of elements $S=\{E^{b_1b_2}, E^{b_1b_3}, E^{b_2b_3}, \\E^{b_1c_2}, E^{b_1c_5}, E^{b_2c_1}, E^{b_2c_4}, E^{b_3c_3}, E^{b_3c_6}\}$.  Note that, in \cite{RN} Nikkuni shows for a 3-connected graph $G$ that $$rk( \mathrm{L(G)})=\frac{1}{2}\Big(\beta_1^2+\beta_1+4|\mathrm{E(G)}|-\sum_{v\in\mathrm{V(G)}} (deg(v))^2\Big),$$ where $\beta_1$ is the first Betti number of $G$.  So it is expected that $rk (\mathrm{L}(K_{3,3,1}))=9.$  

Now, if we consider the coboundary for elements $V^{b_1*}$ we find

\[
\begin{array}{ccc}
\delta^1(V^{b_12}) & = & E^{b_1c_2}+ E^{a_2b_1}- E^{b_1b_3}\\
\delta^1(V^{b_13}) & = & E^{b_1b_2}-E^{a_3b_1}- E^{b_1c_2} \\
\delta^1(V^{b_15}) & = & E^{b_1c_5}+ E^{b_1b_3}-E^{a_5b_1} \\
\delta^1(V^{b_16}) & = & E^{a_6b_1}-E^{b_1b_2}- E^{b_1c_5}. 
\end{array}
\]
So the elements $E^{a_ib_1}$ (for $i$ such that $b_1\cap a_i=\emptyset$) can all be expressed as linear combinations elements of $S$.  This is consistent with the additional relation given by $\delta^1(V^{b_1A}).$  Similarly, all those elements of the form $E^{a_ib_2},$ and  $E^{a_ib_3}$ (for appropriate $a_i$) can be expressed as linear combinations elements of $S$.  Next, if we consider the coboundary for elements $V^{a_1*}$ we find

\[
\begin{array}{ccl}
\delta^1(V^{a_12}) & = & E^{a_1c_2}- E^{a_1b_3} \\
\delta^1(V^{a_14}) & = & E^{a_1c_4}- E^{a_1c_3} \\
\delta^1(V^{a_16}) & = & -E^{a_1c_5}- E^{a_1b_2} \\
\delta^1(V^{a_13}) & = & E^{a_1b_2}+ E^{a_1c_3} E^{a_1c_2} \\
\delta^1(V^{a_15}) & = & E^{a_1c_5}+ E^{a_1b_3}- E^{a_1c_4}.
\end{array}
\]
Thus, all of the elements of the from $E^{a_1c_i}$ (for $i$ such that $a_1\cap c_i=\emptyset$) can be expressed as a linear combination of $E^{a_1b_2}$ and $E^{a_1b_3},$ which can in turn be expressed as a linear combination of the elements in $S$.  Similarly, those elements of the form  $E^{a_jc_i}$ can be expressed as a linear combination of $E^{a_lb_k}$ for those $l$ and $k$ such that $a_l\cap b_k=\emptyset$.  Finally, if we consider the coboundary for elements $V^{c_1*}$ we find

\[
\begin{array}{ccc}
\delta^1(V^{c_13}) & = & E^{c_1c_3}+ E^{b_2c_1}- E^{a_3c_1} \\
\delta^1(V^{c_14}) & = & E^{a_4c_1}+ E^{c_1c_4}- E^{c_1c_3} \\
\delta^1(V^{c_15}) & = & E^{c_1c_5}- E^{c_1c_4}- E^{a_5c_1} \\
\delta^1(V^{c_16}) & = & E^{a_6c_1}- E^{c_1c_5}- E^{b_2c_1}. 
\end{array}
\]
So the elements $E^{c_1c_i}$ (for $i$ such that $c_1\cap c_i=\emptyset$) can be written as a linear combination of $E^{c_1b_2}$ and $E^{a_jc_1}$ (for $j$ such that $a_j\cap c_1=\emptyset$), which can be written as linear combinations of those elements in $S$.  Similarly, all the remaining elements, $E^{c_ic_j},$ can be written as linear combinations of the elements in $S$.  Thus completing our proof.  
\end{proof}

We will make use of two relations that are known for the Wu invariant of $K_{3,3}$.  The Wu invariant
of $f(K_{3,3})$ can be expressed in this simple combinatorial form \cite{T}:
$$\mathcal{L}(f)=\sum_{(x,y)}\varepsilon(x,y)l(f(x),f(y)),$$
the sum over all unordered disjoint pairs of edges in $G,$ where $l(f(x),f(y))$ is the sum of the signs of the crossing between $f(x)$ and $f(y)$, and $\varepsilon(x,y)$ is a weighting defined, 

$$
\varepsilon(x,y) = 
\begin{cases} 
-1,  & \mbox{for }(c_i,b_l) \mbox{ if }i\mbox{ is odd}\\
1, & \mbox{else} 
\end{cases}
$$
where the edges of $K_{3,3}$ are labeled as indicated in Figure \ref{K33}.  
There is another invariant known as the $\alpha-invariant$ of $f,$ \cite{Ta}, for a spatial embedding of $K_{3,3}$ it is as follows:
$$\alpha (f)=\sum_{\gamma\in\Gamma_H}a_2(f(\gamma))-\sum_{\gamma\in\Gamma_4}a_2(f(\gamma)).$$
There is the following relationship between these two invariants:
\begin{prop} \label{alpha-wu}\textup{\cite{MT}} Let $f$ be a spatial embedding of $K_{3,3}$ then, $$\alpha(f)=\frac{\mathcal{L}(f)^2-1}{8}.$$
\end{prop}

The following lemma is about the relationship between the sum of the square of the linking number of all of the links in $K_{3,3,1}$ and the sums of the squares of the Wu invariant of $K_{3,3}$ subgraphs and $K_{3,3}$ subdivisions.  Let the valence 9 vertex of $K_{3,3,1}$ be labeled $A$.  Let $G_i$ for $i=1,\dots, 18$ be the subdivisions of $K_{3,3}$ obtained by deleting three of the edges adjacent to $A$ and then deleting the two edges not adjacent to those already deleted edges, see Figure \ref{Gi}.  Let $H_i$ for $i=1,\dots, 6$ be the $K_{3,3}$ subgraphs that are obtained by deleting one vertex $v\neq A$ and deleting two additional edges that are adjacent to $A$, see Figure \ref{Hi}.  Let $K$ be the $K_{3,3}$ subgraph obtained by deleting the vertex $A$.  

\begin{lem}\label{sum}
 For any embedding $f$ of $K_{3,3,1}$ into $\R^3$ or $S^3$ the following holds 
$$\sum_{\gamma\in\Gamma_{3,4} (K_{3,3,1})} lk(f(\lambda))^2=\frac{1}{8}\sum_{G_i} \mathcal{L}(f|_{G_i})^2-\frac{1}{2}\sum_{K} \mathcal{L}(f|_K)^2-\frac{1}{8}\sum_{H_i} \mathcal{L}(f|_{H_i})^2,$$ where $G_i, K, H_i$ are the above described subgraphs.  
\end{lem}
\begin{figure}[h]
\begin{center}
\includegraphics[scale=0.43]{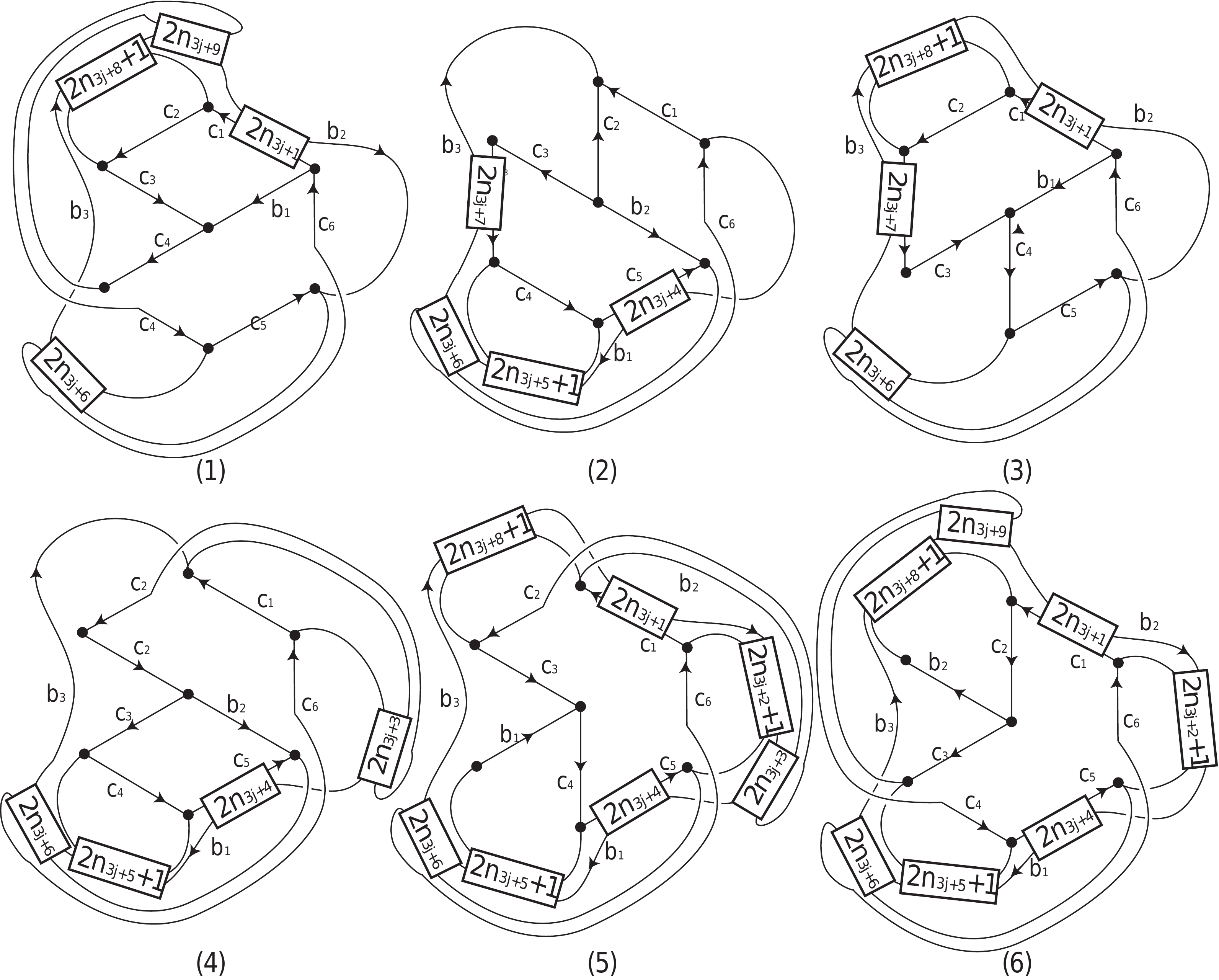}
\caption{The $G_i$ subgraphs of $h(K_{3,3,1}).$}\label{Gi}
\end{center}
\end{figure}
\begin{figure}[h]
\begin{center}
\includegraphics[scale=0.5]{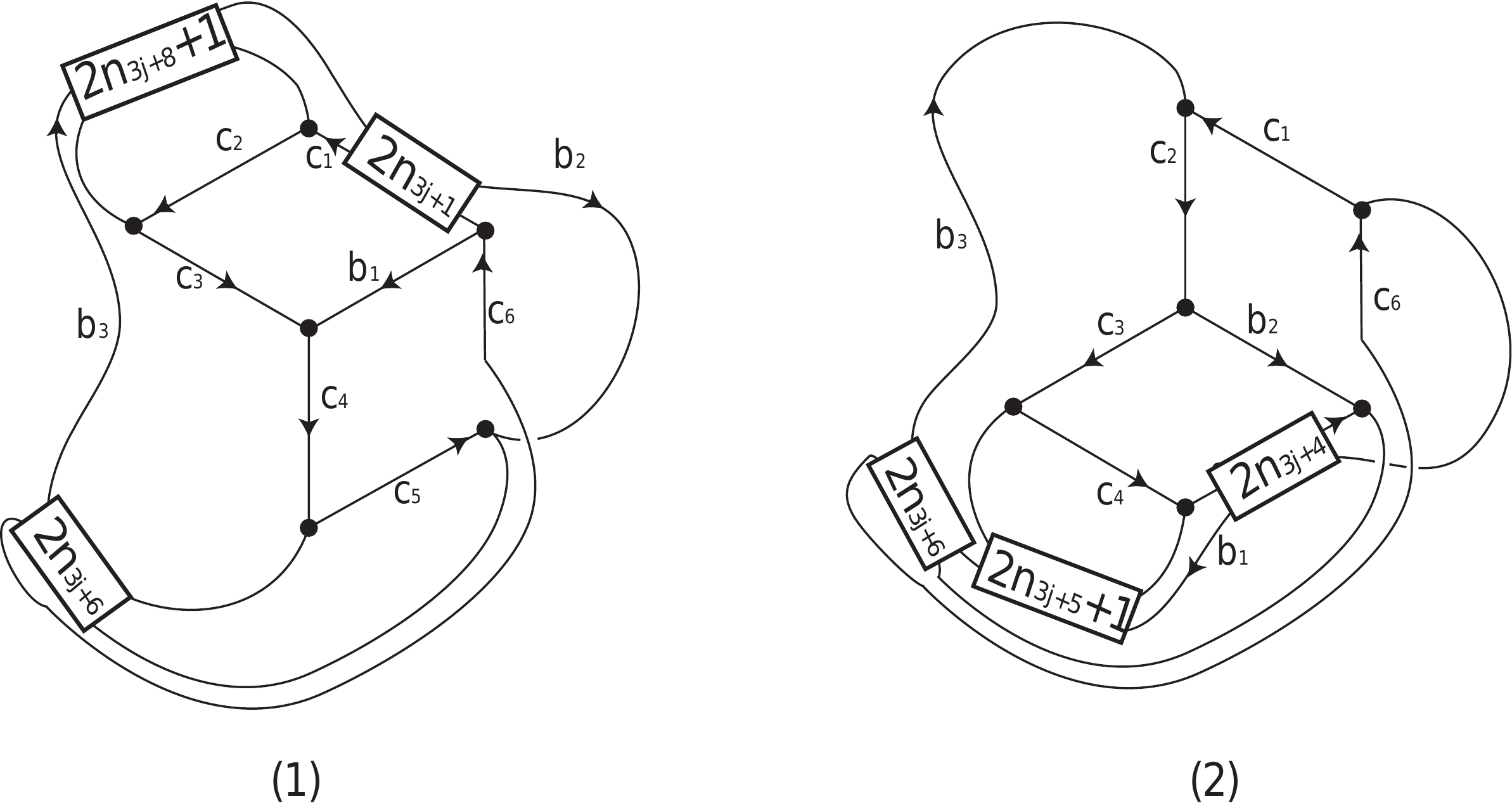}
\caption{The $H_i$ subgraphs of $h(K_{3,3,1}).$}\label{Hi}
\end{center}
\end{figure}

\begin{proof}  From Proposition \ref{wu331} we know there are nine integers $n_1, n_2,\dots, n_9$ such that $f$ is spatial graph-homologous to the embedding $h$ of  $K_{3,3,1}.$  If two embeddings are spatial graph-homologous then they are also spatial graph-homologous when restricted to subgraphs.  Both linking number and the Wu invariant are spatial graph-homology invariants.  Thus we need only show:
$$\sum_{\Gamma_{3,4} (K_{3,3,1})} lk(h(\lambda))^2=\frac{1}{8}\sum_{G_i} \mathcal{L}(h|_{G_i})^2-\frac{1}{2}\sum_{K} \mathcal{L}(h|_K)^2-\frac{1}{8}\sum_{H_i} \mathcal{L}(h|_{H_i})^2.$$
Let $h(G_1),\dots, h(G_{18})$ be as indicated in Figure \ref{Gi}, the subscripts of the $n$'s should be taken modulo 9, $h(G_i)$ for $i=1,2,3$ is as in Figure \ref{Gi}(1) with $j=i \mod 3$, $h(G_i)$ for $i=4,5,6$ is as in Figure \ref{Gi}(2) with $j=i \mod 3$, $h(G_i)$ for $i=7,8,9$ is as in Figure \ref{Gi}(3) with $j=i \mod 3$, $h(G_i)$ for $i=10,11,12$ is as in Figure \ref{Gi}(4) with $j=i \mod 3$, $h(G_i)$ for $i=13,14,15$ is as in Figure \ref{Gi}(5) with $j=i \mod 3$,  and $h(G_i)$ for $i=16, 17,18$ is as in Figure \ref{Gi}(6) with $j=i \mod 3.$  Let $h(H_1),\dots, h(H_{6})$ be as indicated in Figure \ref{Hi}, the subscripts of the $n$s should be taken modulo 9, $h(H_i)$ for $i=1,2,3$ is as in Figure \ref{Hi}(1) with $j=i \mod 3$, $h(H_i)$ for $i=4,5,6$ is as in Figure \ref{Hi}(2) with $j=i \mod 3$.  So the Wu invariants are as follows, where all subscripts are taken modulo 9:
$$
\begin{array}{rlc}
  \mathcal{L}(h|_{G_i})&=2(n_{3i+1}+n_{3i+6}+n_{3i+8}+n_{3i+9})+1  &\text{for }i=1,2,3    \\
   \mathcal{L}(h|_{G_i})&=2(n_{3i+4}+n_{3i+5}+n_{3i+6}+n_{3i+7})+1  &\text{for }i=4,5,6    \\
 \mathcal{L}(h|_{G_i})&=2(n_{3i+1}+n_{3i+6}+n_{3i+7}+n_{3i+8})+1  &\text{for }i=7,8,9    \\
    \mathcal{L}(h|_{G_i})&=2(n_{3i+3}+n_{3i+4}+n_{3i+5}+n_{3i+6})+1  &\text{for }i=10,11,12   \\
    \mathcal{L}(h|_{G_i})&=2(n_{3i+1}+n_{3i+2}+n_{3i+3}+n_{3i+4}+n_{3i+5}+n_{3i+6}+n_{3i+8})+3  &\text{for }i=13,14,15  \\
 \mathcal{L}(h|_{G_i})&=2(n_{3i+1}+n_{3i+2}+n_{3i+4}+n_{3i+5}+n_{3i+6}+n_{3i+8}+n_{3i+9})+3  &\text{for }i=16,17,18  \\
  \mathcal{L}(h|_{H_i})&=2(n_{3i+1}+n_{3i+6}+n_{3i+8})+1  &\text{for }i=1,2,3    \\
   \mathcal{L}(h|_{H_i})&=2(n_{3i+4}+n_{3i+5}+n_{3i+6})+1  &\text{for }i=4,5,6    \\
    \mathcal{L}(h|_{K})&=2(n_{1}+n_{2}+n_{3}+n_4+n_{5}+n_{6}+n_7+n_{8}+n_{9})+3  &
\end{array}
$$

The links in the embedding $h(K_{3,3,1})$ are in two forms.  There are six of the form shown in Figure \ref{K331links}(1), one for each $i=1,3,4,6,7,9$ and three of the form shown in Figure \ref{K331links}(2), one for each $j=1,2,3,$ again all of the subscripts are taken modulo 9.    Thus, 
$$
\begin{array}{cl}
  \sum_{\Gamma_{3,4} (K_{3,3,1})} lk(f(\lambda))^2=& n_1^2+n_3^2+n_4^2+n_6^2+n_7^2+n_9^2  +(n_2+n_3+n_4+n_5+1)^2 \\
  &+(n_5+n_6+n_7+n_8+1)^2+(n_8+n_9+n_1+n_2+1)^2  
  \end{array}
$$
Together these computations give the desired result.  
 
\begin{figure}[h]
\begin{center}
\includegraphics[scale=0.5]{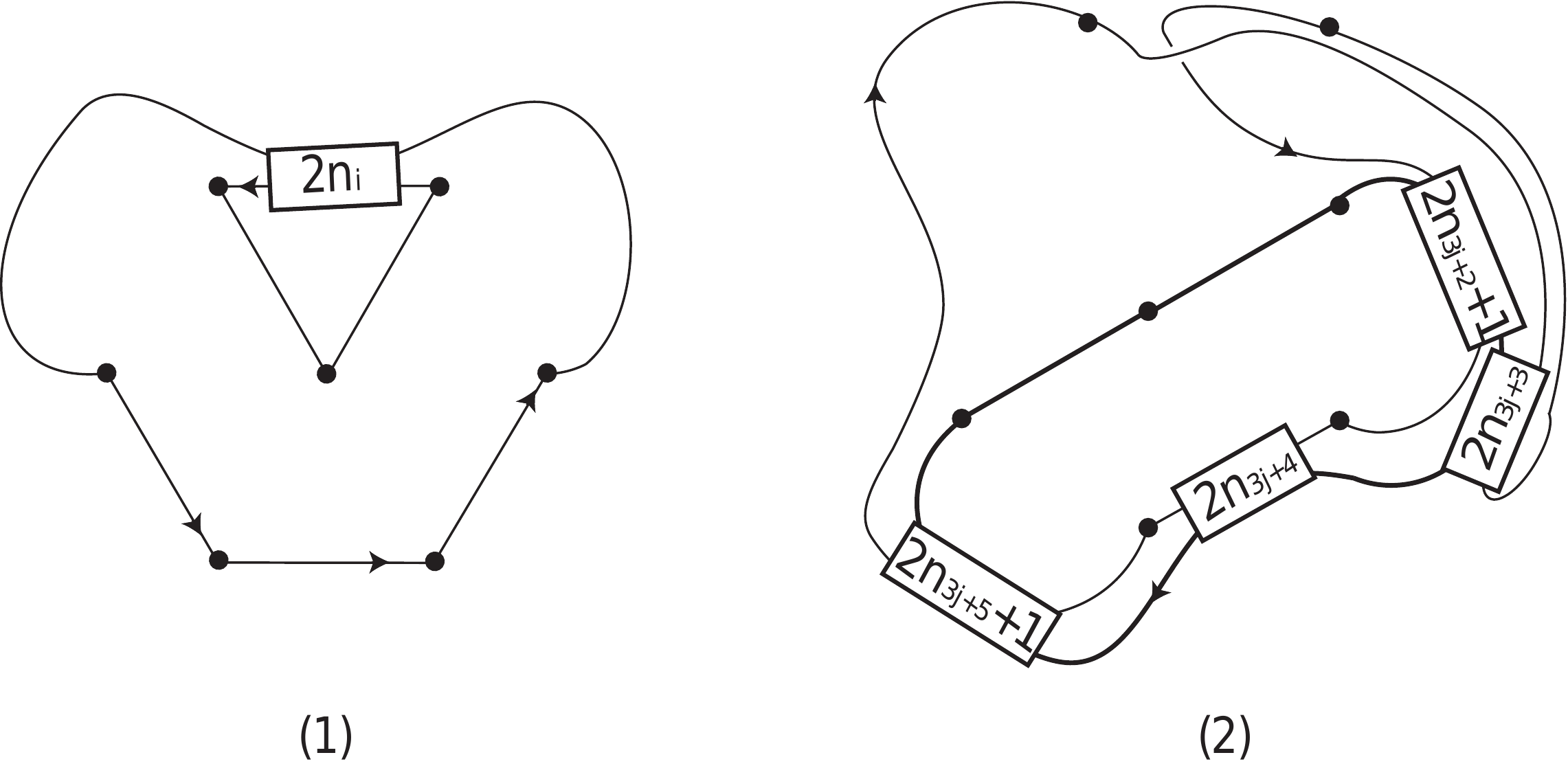}
\caption{The two different types of links found in the embedding $h(K_{3,3,1}).$}\label{K331links}
\end{center}
\end{figure}
\end{proof}

Next we use the relationships between $\mathcal{L}(f)$ and $\alpha(f)$ to obtain the following relationship between the linking number and the second coefficient of the Conway polynomial.  

\begin{thm}\label{k331}  For every embedding $f$ of $K_{3,3,1}$ into $\R^3$ or $S^3$ the following holds
$$\sum_{\lambda\in\Gamma_{3,4} (K_{3,3,1})} lk(f(\lambda))^2=2\Big(\sum_{\gamma\in\Gamma_H}a_2(f(\gamma))-2\sum_{\substack{\gamma\in\Gamma_6\\A\notin\gamma}}a_2(f(\gamma))-\sum_{\substack{\gamma\in\Gamma_5\\A\in\gamma}}a_2(f(\gamma))\Big)+1.$$
\end{thm}

\begin{proof}  Let $f$ be a embedding of $K_{3,3,1}$ into $\R^3$ or $S^3$.  From Lemma \ref{sum} we know, 
$$\sum_{\gamma\in\Gamma_{3,4} (K_{3,3,1})} lk(f(\lambda))^2=\frac{1}{8}\sum_{G_i} \mathcal{L}(f|_{G_i})^2-\frac{1}{2}\sum_{K} \mathcal{L}(f|_K)^2-\frac{1}{8}\sum_{H_i} \mathcal{L}(f|_{H_i})^2.$$  
Then from Proposition \ref{alpha-wu} we see that:  $$\mathcal{L}(f)^2=8\Big(\sum_{\gamma\in\Gamma_H}a_2(f(\gamma))-\sum_{\gamma\in\Gamma_4}a_2(f(\gamma))\Big) +1.$$  
Thus

$$\frac{1}{8}\sum_{G_i} \mathcal{L}(f|_{G_i})^2-\frac{1}{2}\sum_{K} \mathcal{L}(f|_K)^2-\frac{1}{8}\sum_{H_i} \mathcal{L}(f|_{H_i})^2=$$
\begin{eqnarray*}
\Big(\sum_{\substack{\gamma\in\Gamma_H(G_i)\\G_i\in K_{3,3,1}}}a_2(f(\gamma))-\sum_{\substack{\gamma\in\Gamma_4(G_i)\\G_i\in K_{3,3,1}}}a_2(f(\gamma))\Big)
-4\Big(\sum_{\substack{\gamma\in\Gamma_H(K)}}a_2(f(\gamma))-\sum_{\substack{\gamma\in\Gamma_4(K)}}a_2(f(\gamma))\Big) \\
-\Big(\sum_{\substack{\gamma\in\Gamma_H(H_i)\\H_i\in K_{3,3,1}}}a_2(f(\gamma))-\sum_{\substack{\gamma\in\Gamma_4(H_i)\\H_i\in K_{3,3,1}}}a_2(f(\gamma))\Big)
+\frac{18-4-6}{8}.
\end{eqnarray*}
So we need only determine which cycles of $K_{3,3,1}$ are counted in the above sums, and how many times each cycle is counted.  

\noindent\underline{The $G_i$ subgraphs}

\noindent Recall that the $G_i$s are formed by taking $K_{3,3,1}$ and deleting three of the edges adjacent to $A$ and then deleting the two edges not adjacent to those already deleted edges.  This could also be thought of as taking $K_{3,3}$ deleting two adjacent edges and then adding a vertex $A$ and edges from $A$ to each of the vertices that were incident to at least one of the deleted edges.  The $G_i$ are subdivisions of $K_{3,3},$ so some of the Hamiltonian cycles of $G_i$ are Hamiltonian cycles of $K_{3,3,1}$ and some are 6-cycles.  Similarly the 4-cycles will be 5-cycles and 4-cycles in $K_{3,3,1}$.  To count these cycles we will consider different cycles in $K_{3,3,1}$ and determine how many of the $G_i$s contain a given cycle.   

Consider an arbitrary Hamiltonian cycle $\eta$ of $K_{3,3,1}$, to have $\eta$ be in $G_i$ all of the edges of $\eta$ must be in $G_i$.  In particular, the two edges incident to $A$ must be in $G_i$, for this to happen the edge between these two edges, call it $e,$ must be deleted.  In addition, another edge which is not incident to $A$ but is adjacent to $e$ must be deleted, there are two such edges which are not in $\eta$.  Thus two of the eighteen $G_i$ graphs contain $\eta$ as one of their Hamiltonian cycles.  The 6-cycles in $K_{3,3,1}$ can be broken into two sets the ones that contain the vertex $A$ and those that do not.  Since two adjacent edges neither of which are incident to $A$ must be deleted to form a $G_i$, the latter 6-cycle cannot occur.   For a 6-cycle in $K_{3,3,1}$ that contains $A$ the two vertices adjacent to $A,$ call them $v$ and $w$, must be in the same partite set.  Thus the two deleted adjacent edges not incident to $A$ must go between $v$ and $w$.  There is one such way for this to happen, thus each 6-cycle that contains $A$ appears in one of the $G_i$s as a Hamiltonian cycle.  

Every 5-cycle in $K_{3,3,1}$ contains $A$.  To have the edges to the vertex $A$, the edge between the adjacent vertices must be deleted.  As with the Hamiltonian cycles there are two ways to deleted two adjacent edges (not incident to $A$) and delete the said edge.  Thus there are two $G_i$ graphs that contain a given 5-cycle, as a 4-cycle.  Next the 4-cycles of $K_{3,3,1}$ can be put into two groups: 4-cycles that contain $A$ and 4-cycles that do not contain $A$.  By similar reasoning one can see that 4-cycles that contain $A$ will appear in two of the $G_i$s and 4-cycles that do not contain $A$ appear in six of the $G_i$s.  
\smallskip

\noindent\underline{The $K$ subgraph}

\noindent Recall that the subgraph $K$ is the $K_{3,3}$ subgraph obtained by deleting the vertex $A.$  So the Hamiltonian cycles of $K$ are the 6-cycles of $K_{3,3,1}$ that do not contain $A$.  The 4-cycles of $K$ are the 4-cycles of $K_{3,3,1}$ which  do not contain $A$.  

\smallskip
\noindent\underline{The $H_i$ subgraphs}

\noindent Recall that the $H_i$ subgraphs are the $K_{3,3}$ subgraphs that are obtained from $K_{3,3,1}$ by deleting one vertex $v\neq A$ and the two edges that are adjacent to $A$ as well as those vertices in the same partite set as the vertex $v$.  The Hamiltonian cycles of $H_i$ are all be 6-cycles in $K_{3,3,1}$ which contain $A$, as the $H_i$ are $K_{3,3}$ subgraphs with one vertex $v\neq A$ deleted.  Let $c$ be an arbitrary 6-cycle that contains $A$ and does not contain the vertex $v$.  The cycle $c$ will appear in one of the $H_i$s, that is in the $H_i$ which does not contain the vertex $v$.  Next, those 4-cycles that do not contain $A$ will appear in two of the $H_i$, one for each of the vertices that is not $A$ and is not in the said 4-cycle.  In the $H_i$s the vertex $A$ can be thought of as replacing the vertex $v$ that is deleted in the original $K_{3,3}$ subgraph.  Now the 4-cycles that contain $A$, also contain two vertices from one partite set  and one from the partite set that $A$ has now joined.  Thus there are two $H_i$ graphs that contain each 4-cycle.  
\smallskip

All together this gives:
$$\frac{1}{8}\sum_{G_i} \mathcal{L}(f|_{G_i})^2-\frac{1}{2}\sum_{K} \mathcal{L}(f|_K)^2-\frac{1}{8}\sum_{H_i} \mathcal{L}(f|_{H_i})^2=$$
\begin{eqnarray*}
\Big(2\sum_{\gamma\in\Gamma_H}a_2(f(\gamma))+\sum_{\gamma\in\Gamma_6}a_2(f(\gamma))-2\sum_{\substack{\gamma\in\Gamma_5\\A\in\gamma}}a_2(f(\gamma))-2\sum_{\substack{\gamma\in\Gamma_4\\A\in\gamma}}a_2(f(\gamma))-6\sum_{\substack{\gamma\in\Gamma_4\\A\notin\gamma}}a_2(f(\gamma))\Big)\\
-4\Big(\sum_{\substack{\gamma\in\Gamma_6\\A\notin\gamma}}a_2(f(\gamma))-\sum_{\substack{\gamma\in\Gamma_4\\A\notin\gamma}}a_2(f(\gamma))\Big) 
-\Big(\sum_{\substack{\gamma\in\Gamma_6\\A\in\gamma}}a_2(f(\gamma))-2\sum_{\substack{\gamma\in\Gamma_4\\A\notin\gamma}}a_2(f(\gamma))-2\sum_{\substack{\gamma\in\Gamma_4\\A\in\gamma}}a_2(f(\gamma))\Big)\\
+1
\end{eqnarray*}
$$=2\Big(\sum_{\gamma\in\Gamma_H}a_2(f(\gamma))-2\sum_{\substack{\gamma\in\Gamma_6\\A\notin\gamma}}a_2(f(\gamma))-\sum_{\substack{\gamma\in\Gamma_5\\A\in\gamma}}a_2(f(\gamma))\Big)+1,$$
completing our proof.  
\end{proof}

\begin{cor}\label{k331AK}  If an embedding $f$ of $K_{3,3,1}$ is  CA linked then $f$ is knotted.  
\end{cor}

\begin{proof}  If $f(K_{3,3,1})$ is  CA linked then $\sum_{\Gamma_{3,4} (K_{3,3,1})} lk(f(\lambda))^2>1$.  Thus at least one of the $a_2(\gamma)\neq 0$ for $\gamma\in 
\Gamma_H\cup \{\Gamma_6|A\not\in\gamma\}\cup\{\Gamma_5|A\in\gamma\}.$  So $f$ is knotted.    
\end{proof}

\section{Complex algebraically linked and knotted graphs}\label{landk}
In this section we prove our main theorem; given $G\in\mathcal{PF}$ if $f(G)$ is CA linked then $f(G)$ is knotted.  To simplify our discussion we will call a graph, $G$\emph{K-linked} when it has the following property, if an embedding $f$ of $G$ is  CA linked then $f$ is knotted.  So our main theorem can be restated as: All of the graphs of the Petersen family are K-linked.  Before proving this we need the following lemma.  

\begin{figure}[h]
\begin{center}
\includegraphics[scale=0.7]{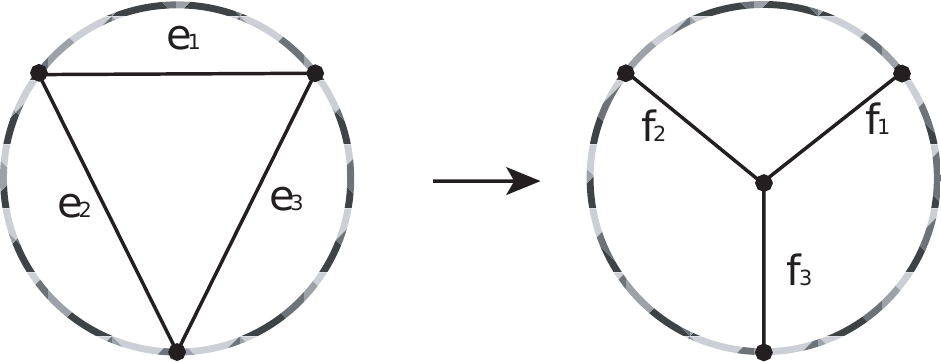}
\caption{The $\nabla$Y-moves.  }\label{TY}
\end{center}
\end{figure}

\begin{lem}\label{triY} Let $G'$ be obtained from $G$ by a $\nabla Y$-move.  If $G$ is K-linked then $G'$ is K-linked.  
\end{lem}

\begin{proof} Let the edges of the triangle $\triangle$ in $G$ be labeled $e_1, e_2,$ and $e_3$, and the edges of the $Y$ in $G'$ be 
labeled $f_1, f_2,$ and $f_3,$ as shown in Figure \ref{TY}.  Let the subgraphs where the two graphs agree be denoted $E$ and $E'$, respectively.  
Define the map $\phi :\Gamma(G)\smallsetminus\triangle\to\Gamma(G'),$ if 
$e_i\not\in\gamma$ then take $\phi (\gamma)$ to be the simple closed curve in $G'$ that is defined by the corresponding edges as in $G$,  if $e_i\in\gamma$ 
then take $\phi (\gamma)$ to be the simple closed curve comprised of the edges that correspond with $\gamma\smallsetminus \{e_i\}$ together with the edges $\{f_i,f_{i+1}\},$ if $\{e_i,e_{i+1}\} \in\gamma$ then $\phi (\gamma)$ is the simple closed curve comprised of the edges that correspond with $\gamma\smallsetminus\{e_i,e_{i+1}\}$ together with the edges $\{f_i,f_{i+2}\}.$ Notice that $\phi$ is surjective.  
Define the map $\psi :\Lambda(G')\to\Lambda(G),$ if $\lambda\in E'$ then $\psi(\lambda)$ is the link consisting of the corresponding edges in $E$, if not there are two edges $f_i, f_{i+1}\in\lambda$ then $\lambda$ is mapped to the link that is comprised of the edges the correspond to $\lambda-\{f_i, f_{i+1}\}$ and the edge $\{e_i\}$.  

Now consider an embedding $f$ of $G'$ which is CA-linked.  Define an embedding $\bar{f}$ of $G$,  where $\bar{f}(E)=f(E')$  and $\triangle$ is mapped onto a tubular neighborhood of $f(Y)\subset f(G')$.  Notice that $\bar{f}(\triangle)$ bounds an embedded disk.  As an abuse of notation we will call the maps on the embeddings of $G$ and $G'$ that result from the maps $\phi$ and $\psi$ by the same names.  Since $f(G')$ is CA-linked there are some number of links $L_1,\dots,L_n\in f(G')$ with nonzero linking number.  Now $\psi(L_i)=L_i$ for all $i$ thus $\bar{f}(G)$ is also CA-linked.  By assumption this implies that $\bar{f}(G)$ is knotted.  Thus there is some simple closed curve $\gamma\in\bar{f}(G)$ which is nontrivially knotted.  Next, $\phi(\gamma)=\gamma$.  So $f(G')$ is knotted.  Therefore $G'$ is K-linked.  
\end{proof}

\noindent {\bf Theorem \ref{main}.}   \emph{If $f$ is a CA linked embedding of $G\in\mathcal{PF}$, then $f(G)$ is knotted.  Which can be restated as:  All of the graphs of the Petersen family are K-linked.  }
Gordon

\begin{proof} Let $G\in\mathcal{PF}$.  If $G=K_6$ then for any embedding $f(K_6)$, by Theorem \ref{nikkuni} \cite{Ni} that, $$\sum_{\lambda\in\Lambda(K_6)} lk(f(\lambda))^2= 2\Big(\sum_{\gamma\in\Gamma_H}a_2(f(\gamma))-\sum_{\gamma\in\Gamma_5}a_2(f(\gamma))\Big)+1.$$  If $f(K_6)$ is  CA linked then $\sum_{\lambda\in\Lambda(K_6)} lk(f(\lambda))^2>1$.  Thus at least one of the $a_2(\gamma)\neq 0$ for $\gamma\in 
\Gamma_H\cup \Gamma_5.$  So $f$ is knotted.    Next, if $G=K_{3,3,1}$ then $G$ is K-linked by Corollary \ref{k331AK}.  If $G\neq K_6$ or $K_{3,3,1}$ the $G$ can be obtained from $K_6$ or $K_{3,3,1}$ by a series of $\nabla$Y-moves, see Figure \ref{PF}.  Thus by Proposition \ref{triY}, $G$ is K-linked.  
\end{proof}

\section{Examples}\label{eg}
\begin{figure}[h]
\begin{center}
\includegraphics[scale=0.7]{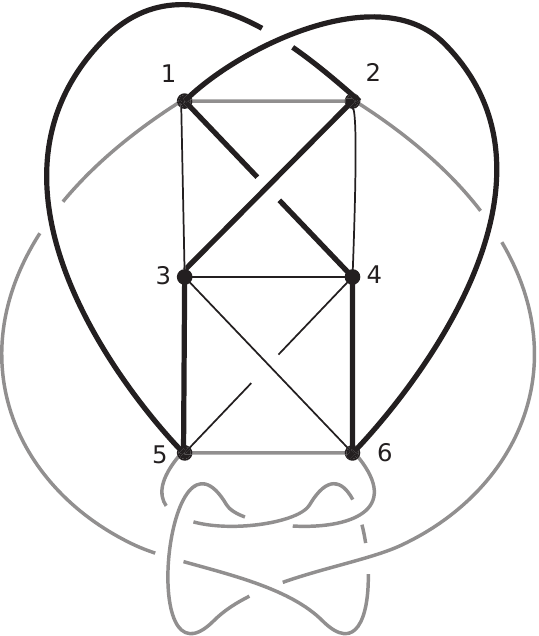}
\caption{A knotted embedding of $K_6$ which contains a single nontrivial link, shown in bold.  }\label{K6}
\end{center}
\end{figure}

In this section we consider the two questions about embeddings of the graphs of the Petersen family:  If $f$ is knotted can that imply a level of complexity in the linking?  If an embedding is not CA linked but contains more than one link or contains a link that is not the Hopf link would this imply the embedding is knotted?  First it should be noted that every embedding of $K_6$ must contain an odd number of links with odd linking number and an even number of links with even linking number, this is a consequence of Conway and Gordon \cite{CG}, where they showed that $$\sum_{\lambda\in\Lambda} lk(\lambda)=1 \mod 2.$$

\begin{figure}[h]
\begin{center}
\includegraphics[scale=0.5]{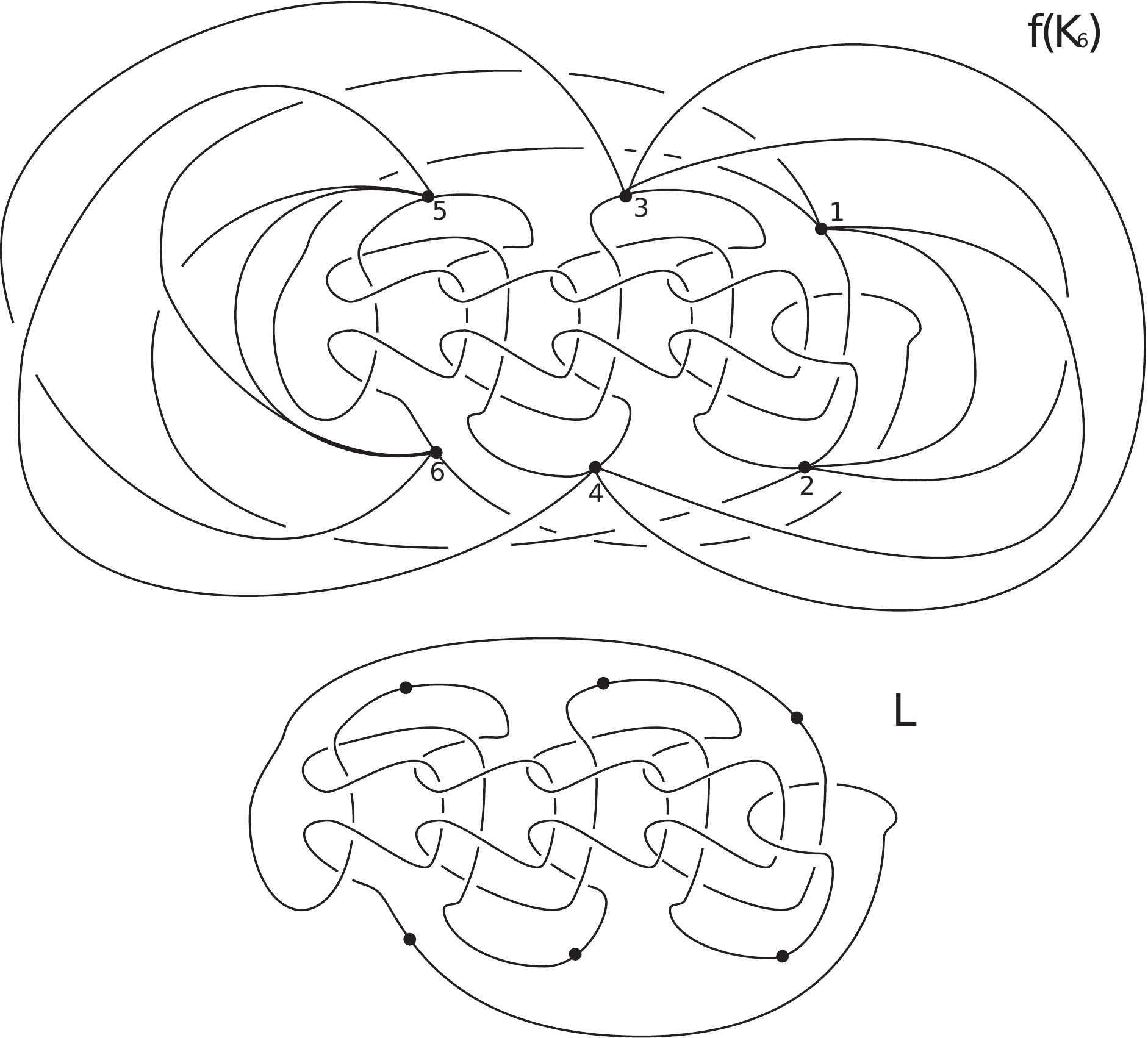}
\caption{The embedding $f(K_6)$ which contains both a Hopf link ($146\cup235$) and the link $L$, but is not knotted.  The link $L.$  }\label{K6I}
\end{center}
\end{figure}
\begin{figure}[h]
\begin{center}
\includegraphics[scale=0.5]{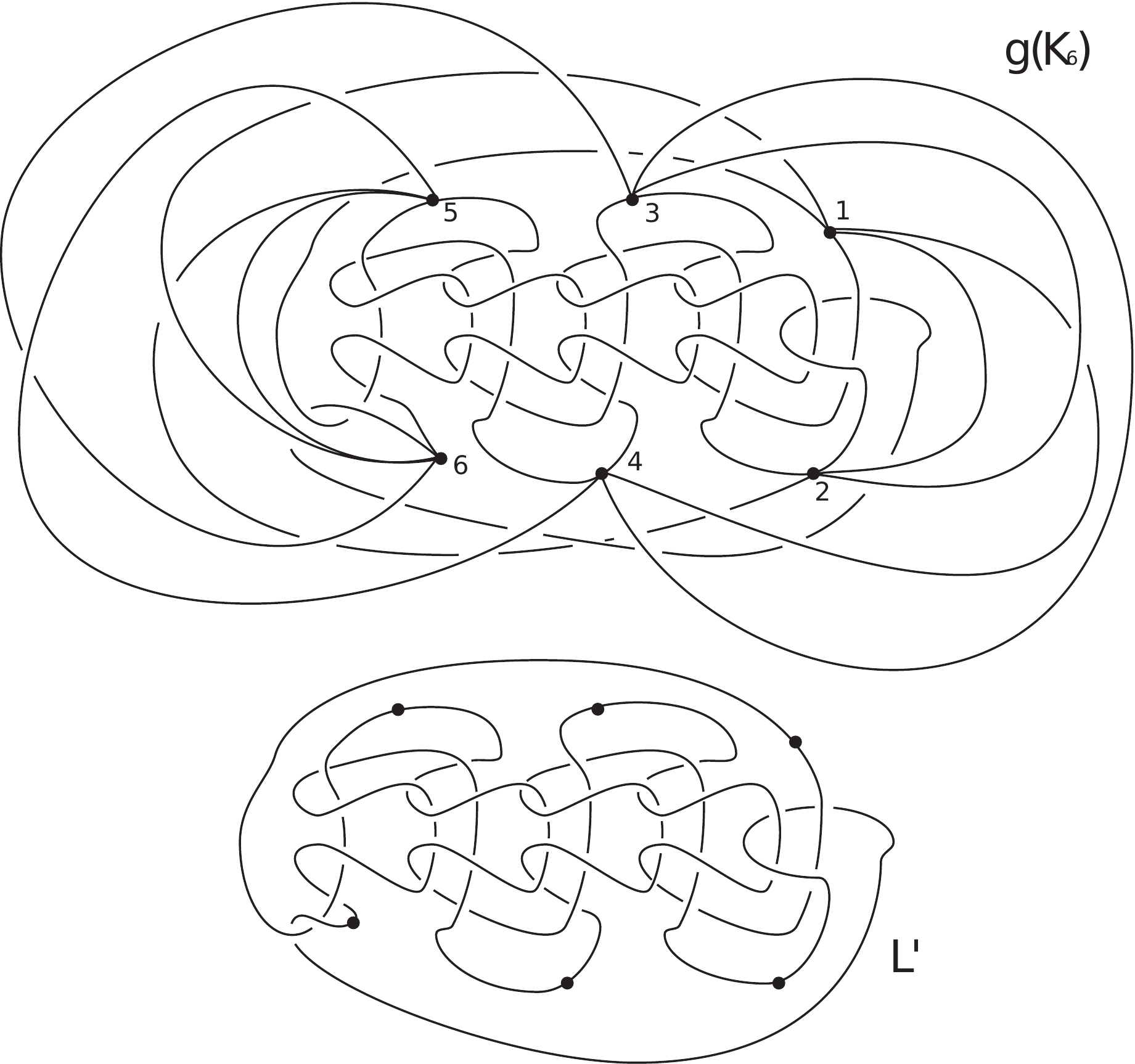}
\caption{The embedding $g(K_6)$ which contains link $L'$, but is not knotted.  The link $L'$. }\label{K6II}
\end{center}
\end{figure}

Now consider the embedding of $K_6,$ shown in Figure \ref{K6}.  This spatial graph contains a single nontrivial link in the pair of cycles 146 and 235 (shown in bold) which form a Hopf link.  This can be simply verified by checking the 10 links.  However, it contains a number of knotted cycles, many of the knots are the connected sum of two trefoils, an example is the cycle 1265.  So this is an example of a spatial graph that is knotted but does not contain any more complicated linking than a single Hopf link.  Thus having a knotted embedding does not imply any increased complexity in the linking.

Next, we will look at two embeddings of $K_6$ which are not CA linked but contain links other than the Hopf link.  The embedding $f(K_6),$ shown in Figure \ref{K6I} contains a Hopf link in the cycles 146 and 235, and the nontrivial link $L$ with linking number 0 (shown in Figure \ref{K6I}) in the cycles 135 and 246.  
\begin{ob} The spatial graph $f(K_6)$ is not knotted.  
\end{ob}
The embedding $f(K_6)$ can be obtained from the embedding in Figure \ref{K6}, by replacing the link $135\cup 246$ with the link $L,$ where $L$ is placed below the other edges.  Notice that all of the knotted cycles in the spatial graph in Figure \ref{K6} contain the edges 15 and 26.  To see this, notice that there are only three crossings that do not involve at least one of the edges 15 or 26, so for there to be a knot without them all of these crossings must be part of the cycle.  But there is only one cycle that contains all of them that is 145236, which is the unknot.  So for there to be a knot in $f(K_6)$ it must contain some of the edges of $L$ because that is where the embeddings differ.  Next the link $L$ is such that if any of the edges is deleted the remaining edges can be isotoped with the vertices fixed and without moving the edges over or around the vertices, so that there are no crossings in the remaining edges.  So the only way to have additional crossings from those edges in $L$ is to have all of them, but together all of the edges make the link $L$.

The second embedding $g(K_6),$ shown in Figure \ref{K6II}, contains a single nontrivial link $L'$ with $|lk(L')|=1$, which is not the Hopf link (shown in Figure \ref{K6II}) in the cycles 135 and 246. In a similar way, in can be seen that $g(K_6)$ is not knotted.  
Gordon
These two examples show embeddings where there is more complex linking but there is not higher linking number, however neither are knotted.  Thus the addition of complexity in these embeddings is not enough to result in a knotted embedding.

\end{document}